\documentclass{article}[14pt]
\usepackage{amsthm,amsfonts,amsmath,amssymb}

\newcommand{\eq}{\begin{equation}}
\newcommand{\en}{\end{equation}}

\newcommand{\giv}{\,|\,}
\newcommand{\prob}{\mathbb P}
\newcommand{\ex}{\mathbb E}

\newcommand{\Nat}{\Bbb N}

\newcommand{\ed}{ \stackrel{d}{=}}

\newcommand{\dd}{{\rm d}}

\newcommand{\1}{{\bf 1}}

\newcommand{\tbar}{\bar{t}}

\def\endpf{\hfill $\Box$ \vskip0.5cm}

\newcommand{\xv}{\mbox{\boldmath$x$}}
\newcommand{\Xv}{\mbox{\boldmath$X$}}
\newcommand{\Yv}{\mbox{\boldmath$Y$}}
\newcommand{\yv}{\mbox{\boldmath$y$}}

\newcommand{\inftyv}{\mbox{\boldmath$\infty$}}

\newtheorem{theorem}{\large Theorem}
\newtheorem{proposition}[theorem] {\large Proposition}

\newtheorem{lemma}[theorem]{\large Lemma}

\begin{document}
\title{Optimal Stopping with Rank-Dependent Loss}
\author{Alexander V. Gnedin\thanks{Postal address:
 Department of Mathematics, Utrecht University,
 Postbus 80010, 3508 TA Utrecht, The Netherlands. E-mail address: gnedin@math.uu.nl}~~~
}
\date{}
\maketitle

\begin{abstract}
\noindent
For $\tau$ a stopping rule adapted to a sequence of $n$ iid observations, we define the loss  to be 
$\ex\,[ q(R_\tau)]$, where $R_j$ is the rank of the $j$th observation, and $q$ is a nondecreasing function of the rank.
This setting covers both the best choice problem with $q(r)={\bf 1}(r>1)$,  and Robbins' problem  with  $q(r)=r$.
As $n\to\infty$ the stopping problem acquires a limiting form 
which is associated with the planar Poisson process.
Inspecting the limit    we establish bounds on the stopping value and
reveal qualitative features of the optimal rule. In particular, 
we show that the complete history dependence  persists in the limit, 
thus answering a question asked by
Bruss \cite{Bruss1} in the context of Robbins' problem.

\end{abstract}
\vskip0.2cm
\noindent
{\large Keywords:  optimal stopping, Robbins' problem, best-choice problem,  planar Poisson process}\\ 
\vskip0.2cm
\noindent
\large{2000 Mathematics Subject Classification: Primary 60G40, Secondary 60G70}

\vskip0.5cm
\noindent
{\bf 1. Introduction}
Let $X_1,\ldots,X_n$  be  a sequence of iid observations, 
sampled from the uniform distribution on $[0,1]$
(in the setup of this paper this assumption covers the general case 
of arbitrary continuous distribution). 
For $j\in [n]:=\{1,\ldots,n\}$
define  {\it final ranks}  as
$$R_{j}=\sum_{k=1}^n \1(X_k\leq X_j),$$
so $(R_1,\ldots,R_n)$ is an equiprobable  permutation of $[n]$.
Let $q:\Nat\to{\mathbb R}_+$ be a nondecreasing loss function with $q(1)<q(\infty):=\sup q(r)$.
In `secretary problems' \cite{SamSurv} one is typically interested in the large-$n$ behaviour of 
the minimum risk 
\eq\label{min}
V_n({\cal T}_n)=  \inf_{\tau\in {\cal T}_n} \ex[q(R_{\tau})],
\en
where  ${\cal T}_n$ is a given class of stopping rules with values in $[n]$.
Two classical loss functions are
\begin{itemize}
\item[(i)]
$q(r)=\1(r>1)$, for the best-choice problem
of maximising the probability of stopping at the minimum observation $X_{n,1}:=\min(X_1,\ldots,X_n)$,  
\item[(ii)]
$q(r)=r$, for 
 the problem of minimising
the expected rank. 
\end{itemize}

\par Many results are available for the case where ${\cal T}_n$  in (\ref{min})
is the class ${\cal R}_n$ of {\it rank rules}, which are  the stopping rules adapted 
to the  sequence of {\it initial ranks}
$$I_{j}=\sum_{k=1}^j \1(X_k\leq X_j)=\sum_{k=1}^j \1(R_k\leq R_j) ~~~~~~~(j\in [n]),$$
see \cite{FrankSam, Gianini, GianSam}. By independence of the initial ranks,
the optimal decision to stop at the $j$th observation depends  only on $I_j$.
The limiting risk
$V_\infty({\cal R}):=\lim_{n\to\infty}V_n({\cal R}_n)$
has interpretation in terms of 
a continuous-time stopping problem \cite{GianSam}.
Explicit formulas for  $V_\infty({\cal R})$ are known in some cases, 
for bounded and unbounded $q$, 
including the two classical loss functions and their generalisations \cite{BG, CMRS, FrankSam, Mucci1, Mucci2}.

\par  Much less explored are the problems where ${\cal T}_n$ is the class ${\cal F}_n$ of all stopping rules adapted to the natural 
filtration $(\sigma(X_1,\ldots,X_j),~j\in [n])$.
The principal difficulty here is that, for general $q$, the decision to stop on $X_j$ 
must depend not only on $X_j$ but also on the full vector $(X_{j-1,1},\ldots,X_{j-1,j-1})$ of order statistics of $X_1,\ldots,X_{j-1}$. In this sense, 
the optimal rule is {\it fully history-dependent}. 
Specifically, the ${\cal F}_n$-optimal rule has the form
\eq\label{optF}
\tau_n=\min\{j: X_j<h_{j}(X_{j-1,1},\ldots,X_{j-1,j-1})\}
\en
(with $h_{n,1}={\rm const},~h_{n,n}=1$),
where $(h_{n,j},~j\in [n])$ is a collection of functions with certain monotonicity properties. 
The dependence on history is reducible to the first $m-1$ order statistics if $q$ is {\it truncated} at $m$:
$q(r)=q(m)$ for $r\geq m$, but even then the analytical difficulties are severe.
The asymptotic value $V_\infty({\cal F}):=\lim_{n\to\infty}V_n({\cal F}_n)$ is known
explicitly only for the best-choice problem (hence for any $q$ truncated at $m=2$),
see \cite{BC} for the formula and history. 
{\it Robbins' problem} is the problem (\ref{min}) with ${\cal T}_n={\cal F}_n$ and the linear loss function $q(r)=r$,
see \cite{Assaf, Bruss1,  Bruss2, Bruss3}.

\par  The full history dependence makes  explicit analysis of the ${\cal F}_n$-optimal rule  hardly possible, thus
it is natural to seek for tractable  smaller classes of rules, with some kind of reduced dependence on the history.
Of course, the rank rules is one of such classes, and the optimal rule in ${\cal R}_n$ is also of the 
form (\ref{optF}), with the special feature that 
$h_{n,j}(x_1,\ldots,x_{j-1})=x_{\iota_n(j)}$ (for $x_0:=0\leq x_1\leq \ldots\leq x_{j-1}\leq 1$ and $j>1$), 
where $\iota_n(j)\in\{0,\ldots, j-1\}$ 
is some threshold value of $I_j$,
and $h_{n,1}=0$.
Another interesting possibility is to consider the  class
${\cal M}_n$ of {\it memoryless rules} of the form 
\eq\label{mless}
\tau=\min\{j:~X_j\leq f_j\},
\en
where  $(f_{n,j},~j\in [n])$ is an increasing sequence of thresholds. These rules are again of the form (\ref{optF}), this time 
with constants in the role 
functions $h_{n,j}$.
By familiar monotonicity arguments (which we recall in Section 4)
the limiting value
$V_\infty({\cal M}):= \lim_{n\to\infty} V_n({\cal M})$ (finite or infinite)  exists for arbitrary $q$.
See \cite{HillK, RubinS} for other classes of stopping rules with restricted dependence on history.

\par Memoryless rules were intensively studied in the context of 
Robbins' problem, in which case they outperform, asymptotically, the rank rules, meaning that 
$V_\infty({\cal M})<V_\infty({\cal R})$, see \cite{Assaf, Bruss2, Bruss3}.
In a recent survey of Robbins' problem Bruss \cite{Bruss1} stressed that a principal further step would be to either prove or disprove
that $V_\infty({\cal F})<V_\infty({\cal M})$.
Coincidence of the asymptotic values
$V_\infty({\cal F})=V_\infty({\cal M})$
 would imply that history dependence of the overall optimal rule were negligible, meaning that
deciding about some $X_j$ one should essentially focus on the current observation alone.

\par In this paper we extend the approach in 
\cite{GnFI, BC, LastR, Denis} by establishing that
the stopping problem in ${\cal F}_n$
has a limiting `$n=\infty$' form based on the planar Poisson process.
The interpretation of limit  risks in terms of the infinite model 
makes obvious the inequality
 $V_\infty({\cal F})<V_\infty({\cal M})$ for any $q$ provided the values are finite, 
which is true for both  the best-choice problem and Robbins' problem.
Thus the complexity does not disappear in the limit, and the full history dependence persists.
The finiteness is guaranteed  
 if $q(r)$ does not grow too fast, e.g.  $q(r)<c \exp(r^{\beta})$ ($0<\beta<1$) is enough.
In connection with  Robbins' problem, the limiting form 
was reported by the author at the INFORMS Conference on Applied Probability (Atlanta, 14-16 June 1995),
although the Poisson embedding had been exploited earlier 
 \cite{BrussRo}  in the analysis of rank  rules.
See \cite{KR} for a similar development in the problem of minimising $\ex\,[X_\tau]$.

\vskip0.5cm

\noindent
{\bf 2. A model based on the planar Poisson process}
Throughout we shall use the notation $\overline{\Nat}=\Nat\cup\{\infty\}$, and $\overline{\mathbb R}_+=[0,\infty]$ 
 for the compactified halfline. 
\par Let $\cal P$ be the scatter of atoms of  a
homogeneous Poisson point process in the strip $[0,1]\times \overline{\mathbb R}_+$,
with the intensity measure being the Lebesgue measure ${\rm d}t{\rm d}x$.
The infinite collection of atoms can be labelled
 $(T_1,X_{1,1}), (T_2,X_{1,2}),\ldots$  by increase of the second component.
Thus $\Xv_1:=(X_{1,1},X_{1,2},\ldots)$ is the increasing sequence of points of a unit Poisson process
on ${\mathbb R}_+$, the $T_r$'s are iid uniform $[0,1]$, and $\Xv_1$ and $(T_r, ~r=1,2,\ldots)$ are independent.
An atom  $(T_r,X_{1,r})\in {\cal P}$ will be understood 
as observation with value  $X_{1,r}$,  arrival time  $T_r$ and  
{\it final rank}  $r$. We define 
the {\it initial rank} of $(T_r,X_{1,r})$ as one plus the number of atoms in the open rectangle $]0,T_r[\,\times\,]0,X_{1,r}[$.
Note that the coordinate-wise ties among the atoms
only have probability zero. 

\par
To treat in a unified way both finite and infinite point configurations in the strip, we introduce
the  space $\cal X$ of all
 nondecreasing nonnegative sequences $\xv=(x_1,x_2,\ldots)$ where $x_r\in \overline{\mathbb R}_+$,
with the convention that 
a sequence with finitely many proper terms is always padded by infinitely many terms $\infty$.
In particular, the sequence $\varnothing:=(\infty,\infty,\ldots)$ is the sequence with no 
 finite terms.
The space $\cal X$ is endowed with the product topology inherited from  $\overline{\mathbb R}_+^\infty$.
We denote
$\xv\cup x$ the nondecreasing sequence obtained by inserting $x\in \overline{\mathbb R}_+$ in $\xv$, with understanding that
$\xv\cup \infty=\xv$. 
A strict partial order on $\cal X$ is defined by setting
$\xv\prec\yv$ if $x_r\leq y_r$ for $r=1,2,\ldots$ with at least one of the inequalities strict.
Clearly, $ \xv\cup x\prec \xv$ for $x<\infty$.

\par We regard $\Xv_1$ as the terminal state of a $\cal X$-valued process $(\Xv_t, t\in [0,1])$, 
where $\Xv_t$ is obtained by removing the entries $X_{1,r}$ of $\Xv_1$ with $T_{r}>t$.
Clearly,  $\Xv_t$ is an increasing sequence of atoms of 
 a Poisson  process on ${\mathbb R}_+$ with intensity measure 
$t{\rm d}x$.  
For $t\in \{T_r\}$ let $X_t, R_t, I_t$ be the value, the final rank and the initial rank
of the observation arrived at time $t$, respectively, and for $t\notin \{T_r\}$ let $X_t= R_t= I_t=\infty$. 
We have $\Xv_t=\Xv_{t-}\cup X_t$, so
$\Xv_t=\Xv_{t-}$
unless $t\in\{T_r\}$.

\par The process $(\Xv_t,~t\in[0,1])$ is Markovian, with right-continuous paths, the initial state $\Xv_0=\varnothing$ and the
jump-times $\{T_r\}$ which comprise a dense subset of $[0,1]$. 
Each component $(X_{t,i},~t\in [0,1])$ is a nonincreasing process, which satisfies
$X_{0+,i}=\infty$   and changes its value at every 
{\it $i$-record} (observation of initial rank $i$).
The jump-times of $(\Xv_{t,i},~t\in[0,1])$ are 
the arrival times of $i$-records; these occur 
according to a  Poisson process of intensity 
$t^{-1}\dd t$ independently  for distinct $i\in \Nat$,
as is known from the extreme-value theory.

\par Define a {\it stopping rule}
 $\tau$ to be a variable which may only assume one of the random values  $\{T_r\}\cup\{1\}$,
and  satisfies the measurability condition $\{\tau\leq t\}\in \sigma (\Xv_{s},~s\leq t)$  for $t\in [0,1]$.
The condition says that
the decision to stop not later than $t$ is determined by atoms ${\cal P}\cap([0,t]\times {\mathbb R}_+)$ arrived within the time interval
$[0,t]$.
Such rules are called in \cite[Definition 2.1]{KR} `canonical stopping times'.

\par 
We fix a nondecreasing nonnegative loss function $q$ satisfying $q(1)<q(\infty)$.
The risk incurred by  stopping rule  $\tau$ is assumed to be 
\eq\label{loss}
\ex[q(R_\tau)]=\sum_{r=1}^\infty q(r) ~\prob(\tau=T_r)+q(\infty)~\prob(\tau=1),
\en
where the terminal component is nonzero if and only if $\prob(\tau=1)>0$.
 Let $\cal F$ be the set of all stopping rules, and let
$V({\cal F})=\inf_{\tau\in{\cal F}}\ex[q(R_\tau)]$
be the minimal risk. 

\par The class $\cal R$ of {\it rank rules} is defined by
a more restrictive measurability condition
$\{\tau\leq t\}\in 
\sigma (I_{s},~s\leq t)$  for $t\in [0,1]$. That is to say,
by a rank rule  the information of observer at time $t$ amounts to the collection of arrival times on $[0,t]$ of $i$-records, for 
all $i\in \Nat$.
The optimal stopping problem in ${\cal R}$ is equivalent to 
`the infinite secretary problem' in
\cite{GianSam}.
By \cite[Theorem 4.1]{GianSam}
there exists an optimal rank rule
of the form $\tau=\inf\{t: I_t\leq \iota(t)\}$ ($\inf\varnothing=1$), where $\iota:[0,1[\to \Nat\cup\{0\}$ is a nondecreasing function.
For instance, in the best-choice problem $\iota(t)=\1(t\geq e^{-1})$.

\par  A {\it memoryless} rule is 
a stopping rule of the form
\eq\label{mem-f}
\tau=\inf\{t: X_t\leq f(t)\}  ~~~~~({\rm with~}\inf\varnothing=1),
\en
where $f:[0,1[\,\to {\mathbb R}$ is a  nondecreasing function. 
Denote $\cal M$ the class of memoryless rules, and denote
$V({\cal M})=\inf_{\tau\in{\cal M}}\ex[q(R_\tau)]$ its stopping value. 
 One could consider a larger class of stopping rules by which the decision to stop depends only on the current observation.
However,
the following lemma, analogous to \cite[Lemma 2.1]{Assaf}, shows that such
extension of $\cal M$ does not reduce the risk.

\begin{lemma}\label{Borel} Let $A\subset[0,1]\times{\mathbb R}_+$ be a Borel set.
For  the stopping rule 
$\tau=\inf\{t:~(t,X_t)\in A\}$ 
there exists a memoryless rule 
whose expected loss  is not larger than that of $\tau$.
\end{lemma}
\proof
It is sufficient to consider sets $A$ such that  the area of $A\cap([0,t]\times{\mathbb R}_+)$ is finite for every $t<1$.
Indeed, if the area  of  $A\cap([0,t]\times{\mathbb R}_+)$ is infinite for some $s<1$ then $\tau<s$ a.s.,
hence letting $A'$  to be $A\cap([0,s]\times{\mathbb R}_+)$ shifted by $1-s$ to the right we obtain a rule not worse than $\tau$.
Replace each vertical section of $A$ by an interval adjacent to $0$ of the same length, thus 
obtaining subgraph of a function $g$.
This preserves the distribution of the stopping rule and does not increase the risk, 
by the monotonicity of $q$. 
Break $[0,1]$ into intervals of equal size $\delta$ and approximate
$g$ (in $L^1$) by a right-continuous function $g_\delta$, constant on these intervals.
Suppose  on some adjacent intervals $[t,t+\delta[,~[t+\delta, t+2\delta[$ we have $g_\delta(t)> g_\delta(t+\delta)$.
Let $g_\delta'$ be 
another piecewise constant function with exchanged  values  
on these intervals, $g_\delta(t+\delta)$ and $g_\delta(t)$,
but  outside $[t,\,t+2\delta]$  coinciding with $g$.
Let ${\cal P}'$ be the scatter of atoms obtained by exchanging  the strips
$[t,t+\delta[\,\times {\mathbb R}_+$ and $[t+\delta,t+2\delta[\,\times {\mathbb R}_+$.
Obviously, 
 ${\cal P}'\ed{\cal P}$.
To compare two stopping rules $\tau$ and $\tau'$  defined as in (\ref{mem-f}),  but with $g_\delta$, respectively $g_\delta'$, in place of $f$, we consider 
the selected atom
$(\tau, X_\tau)$ as a function of ${\cal P}$, and consider $(\tau', X_{\tau'})$ as a function of ${\cal P}'$. 
It is easy to see that $X_\tau=X_{\tau'}$ unless $([t+\delta,t+2\delta[\,\times [0,g(t+\delta)])\cap{\cal P}\neq\varnothing$,
whereas in the latter case $X_{\tau'}$ is stochastically smaller than $X_\tau$.
The advantage comes from the event that each of the strips contains an atom below the graph of $g_\delta$. 
It follows that $\tau'$ does better.
Iterating this exchange argument, we see that 
the rule defined by $g_\delta$ is improved by a memoryless rule with
a piecewise constant function. 
Letting $\delta\to\ 0$ shows that one can reduce $A$ to a subgraph of a monotonic $f:[0,1[\to\overline{\mathbb R}_+$.
\endproof

Given  the initial rank $I_t=i$ and the value $X_t=x$ of some observation at time $t$,
the final rank of the atom $(t,x)$ is $i$ plus the number of atoms south-east of $(t,x)$, the latter being 
a Poisson variable with parameter $\bar{t}x$, where and henceforth 
$$\bar{t}:=1-t.$$ 
By  independence properties of $\cal P$,
the adapted loss incurred by  stopping at $(t,x)$ is   equal to $Q(\bar{t}x,i)$, where
\eq\label{Qser}
Q(\xi,i):=
\sum_{r=i}^\infty q(r)\,e^{-\xi}{\xi^{r-i}\over (r-i)!}
\en
For instance, $Q(\bar{t}x,i)=1-e^{-\bar{t}x}\1(i=1)$ 
in the best-choice problem, and
$Q(\bar{t}x,i)=\bar{t}x+i$ in  Robbins' problem.
The formula for $Q$ is  extended for infinite values of the arguments 
as $Q(\cdot,\infty)=Q(\infty,\cdot)=q(\infty)$. 
It is seen from the identity
$${\dd\left[ e^\xi Q(\xi,1)\right]\over\dd\xi^{i-1}}=e^\xi Q(\xi,i)$$
 that the series $Q(\cdot,i)$ have the same convergence radius for all $i$.

\vskip0.5cm
\noindent
{\bf 3. Memoryless rules and finiteness of the risk} 
For $\tau$ a memoryless rule (\ref{mem-f}) with monotone $f$,
 denote $L(f)=\ex\,[q(R_\tau)]$ the expected loss. 
Introduce the integrals
$$
F(t)=\int_0^t f(s)\,\dd s\,, ~~~~S(x)=\int_0^x f^{-1}(y)\,\dd y= xf^{-1}(x)-F(f^{-1}(x))\,,
$$
where $f^{-1}$ is the right-continuous inverse with $f^{-1}(x)=0$ for $x<f(0)$. 
Note that $\prob(\tau>t)=\exp(-F(t))$, and that given $\tau=t<1$ the law of $X_\tau$ is uniform on $[0,f(t)]$.
The formula for the risk follows by  
conditioning on the
location of the leftmost atom below the graph of $f$ and using the fact that
 the  configurations of atoms above the graph  and below it are independent:
\eq\label{m-rule-risk}
 L(f) =\int_0^1 e^{-F(t)}\,\dd t\int_0^{f(t)} Q(\bar{t}x+S(x),1)\,\dd x\,+ e^{-F(1)}\,q(\infty).
\en

\par Assuming that $F(1)=\infty$, so  the terminal part is $0$,  
computation of the first variation of $L(f)$ shows that 
an optimal $f$ must satisfy a rather complicated functional equation:
\begin{eqnarray}\label{varipr}
Q(f(t)-F(t),1)=~~~~~~~~~~~~~~~~~~~~~~~~~~~~~~~~~~~~~~~~~~~~~~~~~~~~~~~~~~~~~~~~~~~~~~~~~~~~~~~~~~~\\ \nonumber
\int_t^1\exp(F(t)-F(s))\dd s\left[\int_0^{f(t)}Q(S(x)+x\bar{s},1)\dd x+\int_{f(t)}^{f(s)} Q(S(x)+x\bar{s},2)\dd x\right].
\end{eqnarray}
\par 
A rough upper bound
\eq\label{m-rule-bound}
 L(f) \leq\int_0^1 e^{-F(t)}\,\dd t\int_0^{f(t)} Q(x,1)\,\dd x\,+ e^{-F(1)}\,q(\infty)
\en
follows from $\bar{t}x+S(x)\leq x$. 

\par The bound (\ref{m-rule-bound}) is computable for the loss functions
\eq\label{qff}
q(r)=(r-1)(r-2)\cdots(r-\ell)~~~~~~ (\ell\in\Nat),
\en
in which case we have a very simple formula $Q(\xi,1)=\xi^\ell$, 
and  (\ref{m-rule-bound}) becomes
$$L(f)\leq (\ell+1)^{-1}\int_0^1 e^{-F(t)} f(t)^{\ell+1}\dd t\,.$$
 Solving the variational problem for $F$ with boundary conditions $F(0)=0, ~F(1)=\infty$, we see that 
the minimal value  of the right-hand side is $(\ell+1)^\ell$, which is  attained by
the function $f(t)=(\ell+1)/(1-t)$.

\par It is instructive to directly analyse the memoryless rules with hyperbolic threshold
$$
f_b(t):={b\over 1-t}\, ~~~~~~(b>0)
$$
and $q$ as in (\ref{qff}).
We calculate  $e^{-F(t)}=(1-t)^b$ and $S(x)=(x-b-b\log(x/b))$ (for $x>f(0)=b$).
For $\ell=1$
integrating by parts in (\ref{m-rule-risk}) we obtain
\eq\label{L1}
L(f_b)={b\over 2}+{1\over b^2-1},
\en
which is finite for 
all $b>1$, with the minimum  $1.3318\cdots$  attained at $b=1.9469\cdots$
(which agrees with \cite[Example 4.2]{Assaf} where the minimum is $2.3318\cdots$
for the linear loss $q(r)=r$).
For $\ell=2$ 
\eq\label{L22}
L(f_b)= {b^3\over 3}+{2(b^4-2b^3+2b^2+6b-4)\over (b-2)(b-1)^2(b+1)(b+2)},
\en 
which is finite for all $b>2$, with minimum $4.4716\cdots$ at $b=2.96439\cdots$.
Formulas become more involved for larger $\ell$, a common feature being that 
$L(f_b)<\infty$ for $b>\ell$.
For $\ell=3$,   the minimum is $24.8061$ at 
$3.9734\cdots$.
For $\ell=4$, the minimum is $194.756\cdots$ at $b=4.979\cdots$.
The upper bound (\ref{m-rule-bound}) becomes
$$L(f_b)< \int_0^1 (1-t)^b\int_0^{b/(1-t)} x^\ell\dd x={b^{\ell+1}\over (\ell+1)(b-\ell)},$$
which attains  minimum at $b=\ell+1$ in agreement with what we have obtained above. 
\vskip0.5cm
\noindent
{\bf Remark.} Notably,
the memoryless rule with threshold $f_{\ell+1}$ is overall optimal in the related 
stopping problem $\ex[(X_\tau)^\ell]\to \inf$, for arbitrary $\ell>0$.
For $\ell=1$ we face here a variant of `Moser's problem' 
associated with $\cal P$
(see
\cite{Assaf, Bruss1, KR} and references therein).

\vskip0.5cm

\par  It can be read from \cite{Bruss1, Assaf, CMRS} that 
for the linear loss $q(r)=r$ we have
$V({\cal M})=\inf L(f)<V({\cal R})=3.8695\cdots$.

The minimiser of $L(f)$
is not known explicitly,
but some approximations to it can be read from
\cite{Assaf} (where they appear in the course of asymptotic analysis 
of the finite-$n$ Robbins' problem). 
We did not succeed to solve (\ref{varipr}) even
for the best choice problem, although there is a simple suboptimal  
 rule with constant threshold $f(t)=1.503\cdots$  achieving
$L(f)=1-0.517\cdots$  (to be compared with the  value $V({\cal F})=1-0.580\cdots$, see \cite[p. 682]{GnFI}) 
hence beating the rank rules: $V({\cal M})<V({\cal R})=1-0.368\cdots$.
\par It would be interesting to know for which $q$ the memoryless rules outperform the rank rules 
and if it is possible, for unbounded $q$, to have the memoryless risk finite while infinite for the rank rules.
We sketch some results in this direction.
 From the above elementary estimates 
$V({\cal M})<\infty$ provided  $q(r)<c\, r^\ell$ for some constants $c>0$, $\ell >0$.
For such $q$ the risk of rank rules is also finite.
Moreover, 
Mucci \cite[p. 426]{Mucci2} showed  that 
for the loss function 
$q(r)=r(r+1)\cdots (r+\ell-1)$ $(\ell\in \Nat)$ 
the minimum risk of rank rules is
$$V({\cal R})=\ell!\,\, 
\prod_{j=1}^\infty \left(1+{\ell+1\over j}\right)^{\ell/(\ell+j)}\,$$
(which extends the $\ell=1$ result from \cite{CMRS}).
For $\ell=2$ the formula yields
$33.260\cdots$, while
the $f_b$-rules do worse, with $\inf_b L(f_b)=38.068\cdots$
(as computed from (\ref{L1}) and (\ref{L22}) using the linearity of $L(f)$ in $q$).


\par In fact, $V({\cal M})<\infty$ for many loss fuctions growing much faster than polynomials.
\begin{proposition} If 
 $q(r)<c \exp(x^\beta)$ for some $c>0$ and $0<\beta<1$ then $V({\cal M})<\infty$.
\end{proposition}
\proof
The risk is finite 
for the memoryless rule with $f(t)=(1-t)^{-\alpha}$ for any $\alpha>(1-\beta)^{-1}$.
To see this, use the bound (\ref{m-rule-bound}) and formulas
$$Q(x,1)=O(\exp(x^\beta))~~~ (x\to\infty), ~~~~\exp(-F(t))=\exp\left(-{1\over (\alpha-1)(1-t)^{\alpha-1}}\right),$$
which also imply that for this rule $\prob(\tau=1)=0$.
Now $\ex[\exp((X_\tau)^\beta))]$ is estimated from
asymptotics of the incomplete gamma function.  
\endpf

However, the risk is infinite for any stopping rule if $q$ grows too fast. 
The following result is an analogue of \cite[Proposition 5.3]{GianSam} 
for rank rules.

\begin{proposition}
If $Q(b,1)=\infty$ for some $b\in {\mathbb R}_+$ then $V({\cal F})=\infty$, i.e. there is no stopping rule $\tau\in{\cal F}$ with finite risk.
\end{proposition}
\proof
Choose any $x$ with $S(x)=x-b-b\log(x/b)>b$.
The conditional loss by stopping above $f_b$ is infinite, thus we can only consider stopping rules $\tau$
which never do that and satisfy $\prob(\tau=1)=0$.
On the other hand, on the nonzero event 
$\{{\cal P}\cap \{(t,y): y<\min(x,f(t))\}=\varnothing\}$  stopping occurs at some atom $(s,z)$ with $s>1-b/x,\,z>x$,
and averaging we see that the expected loss is infinite.
\endpf

\vskip0.5cm\noindent
{\bf Remark}
By \cite[Section 5]{GianSam}, $V({\cal R})=\infty$  if $\sum_r (\log q(r))/r^2=\infty$.
For instance,  the loss structure $q(r)=e^r$
implies that the risk of rank rules is infinite.
It is not known if the risk of rank rules is finite for $q(r)=\exp(x^\beta)$ with $0<\beta<1$.
\vskip0.5cm
\par For the sequel we  assume that the loss function satisfies
\eq\label{Ass}
\limsup{ q(r+1)\over q(r)}=C,
\en
with some constant $C>0$.
The assumption implies that
 $Q(x,i)<\infty$ for all finite $x,i$.
Another consequence is that $\ex[q(R_\tau)]<\infty$ implies $\ex[q(R_\tau+N)]<\infty$ for $N$ either a fixed positive integer 
or a Poisson random variable, independent of $\tau$.

\begin{lemma}\label{cont1} If $\ex[q(R_\tau)\giv \Xv_0=\xv]<\infty$ then $\ex[q(R_\tau)\giv \Xv_0=\xv']$ is finite and continuous
in $x$, where $\xv'$ is either $\xv\cup x$ or  $(x_1+x,x_2+x,\ldots)$.
\end{lemma}
\proof As $x$ changes to some $x'$, the outcome $R_\tau$ can only change if there 
is an atom between $x$ and $x'$, which occurs with probability about $|x-x'|$ when $x,x'$ are close.
Conditionally on this event, the change of expected loss is bounded in consequence of (\ref{Ass}).
\endpf

\vskip0.5cm

\noindent
{\bf 3. Properties of the optimal rule} The
optimal stopping problem in $\cal F$ is a problem of Markovian type, associated with the time-homogeneous Markov process
$((\Xv_t,I_t),~t\in [0,1])$, with state-space ${\cal X}\times \overline{\Nat}$
and time-dependent loss $Q(\bar{t}X_t,I_t)$ for stopping at time $t$.
If $I_t$ assumes some finite value $i$ then $t\in\{T_r\}$ and $X_{t,i}=X_t$, 
which combined with the fact that ranking of the arrivals after $t$ depends on 
${\cal P}\cap([0,t]\times{\mathbb R}_+)$ through $\Xv_t$ shows that
$(\Xv_t,I_t)$ indeed summarises all relevant information
up to time $t$. 
We choose $(\Xv_t,I_t)$ in favour of (probabilistically equivalent)  data $(\Xv_{t-},X_t)$ 
since $x_i$ is well-defined as a function of $(\xv,i)$ even if $\xv$ has repetitions.

\par Following a well-known recipe, we  consider a family of conditional stopping problems parametrised by 
 $(t,\xv)$. This corresponds 
to the class of stopping rules $\tau>t, \,\tau\in{\cal F}$  that operate under the condition  $\Xv_t=\xv$.
The effect of the conditioning is that each $x_r<X_\tau$ contributes one unit to $R_\tau$ in the event $\tau<1\,$. 
The variable $t$ can be eliminated by a change of variables
which exploits  the
{\it self-similarity} of $\cal P$ (a  property which has no analogue in the finite-$n$ setting):
for $t\in \,]0,1[$ fixed,
the affine  mapping  $(s,x)\mapsto ((s-t)/\tbar,x\tbar)$ preserves both the coordinate-wise 
order and the Lebesgue measure, hence
transforms the point process ${\cal P}\cap ([t,1]\times {\mathbb R}_+)$ into a distributional copy of $\cal P$
with the same ordering of the atoms.
Thus we come to the following conclusion:
\begin{lemma}\label{selfsim} The stopping problem from time $t$ on with history 
$\xv$ is equivalent to  the stopping problem starting with
$\Xv_0=\tbar\xv$ at time $0$.
\end{lemma}

\par Let $v(\xv)$ be the minimum risk given $\Xv_0=\xv$. 
The  function $v$, defined on the whole of $\cal X$, satisfies a lower bound
\eq\label{hp}
v(\xv)\geq\sum_{r=1}^\infty q(r)(e^{-x_{r-1}}-e^{-x_r})~~~~~~~~(x_0=0),
\en
which is strict if the series converges
(the bound is a continuous-time  analogue of the  finite-$n$ `half-prophet' bounds in \cite[Lemma 3.2]{Bruss2}).
The bound  follows by observing that $X_\tau$ cannot exceed the smallest value arrived on $[0,1]$.

\par If $V({\cal F})=\infty$ then, of course, $v(\xv)=\infty$  everywhere,
but for arbitrary unbounded $q$ 
there exists a dense in $\cal X$ set of sequences $\xv=(x_r)$  for which $x_r\uparrow\infty$ so slowly that  $v(\xv)=\infty$.
Thus if $q(\infty)=\infty$, the function $v$ is discontinuous at every point where it is finite. 
If $q$ is truncated at $m$, then clearly $v$ depends only on the first $m-1$ components of $\xv$ and satisfies $v(\xv)<q(m)$.
Let ${\bf 0}=(0,0,\ldots)$.

\begin{lemma}\label{cont2} The following hold:
\begin{itemize}
\item[{\rm (i)}] $v(\xv)<\infty$ implies that $v(\xv\cup x)$  is finite and continuous in $x$,
\item[{\rm (ii)}] if  $q(\infty)<\infty$ then $v$ is continuous, and  satisfies $v(\xv)<q(\infty)$ for $x_1>0$.
\item[{\rm (iii)}]   $v(\xv)\to q(\infty)$ as $\xv\to{\bf 0}$.

\end{itemize}
\end{lemma}
\proof Let $\tau$ be  $\epsilon$-optimal under the initial configuration $\xv\cup x$.
Applying $\tau$ under $\xv\cup x'$, Lemma \ref{cont1} implies that 
$v(\xv\cup x')\leq  v(\xv\cup x)+\epsilon$. Changing the roles of $x,x'$ and letting $\epsilon\to 0$ yield (i).
The continuity of $v$ follows directly from (i) if $q$ is truncated at some $m$. The general bounded case
follows by approximation as $m\to\infty$. 
Assertion (iii) can be derived from (\ref{hp}).
\endpf

\begin{lemma}\label{L2}
If $q$ is not truncated then
\begin{itemize}
\item[{\rm (i)}] $Q(x,i)$ is strictly increasing in both $x$ and $i$,
\item[{\rm (ii)}]  $\xv\prec \yv$ implies $v(\xv)< v(\yv)$ provided these are finite,
\end{itemize}
If $q$ is truncated at $m$ and $q(m-1)<q(m)$ then {\rm (i)} is valid only for $i\in [m]$,  $Q(\xv,i)=q(m)=q(\infty)$ for $i\geq m$,
and a counterpart of {\rm (ii)} holds for the order defined on the first $m-1$ components, 
with $v(\xv)<q(m)$ for all $\xv\in {\cal X}$ with $x_{m-1}>0$.
\end{lemma}
\begin{proof} Assertion (i) follows from (\ref{Qser}) and the monotonicity of $q$.
For (ii), observe that $\xv\prec \yv$ implies $\#\{i: x_i< x\}\geq \#\{i: y_i< x\}$ for all $x>0$.
Hence for every rule $\tau$ the stopped final rank under  $\Xv_0=\xv$ cannot increase when the condition is replaced  
by $\Xv_0=\yv$.

\end{proof}

Let $i(\xv,x):=\#\{r:x_r\leq  x\}$ and 
suppose $\xv$ satisfies $0<x_1\leq x_2\leq \ldots \leq \infty$.
  Applying Lemma \ref{L2}, we see that if $q$ is not truncated then the 
function $Q(x, i(\xv,x))$ is strictly increasing in $x$ from $q(1)$ to $q(\infty)$.
If $q$ is truncated at $m$ and $q(m-1)<q(m)$ then $Q(x, i(\xv,x))$ is strictly increasing as $x$ varies from $0$ to $x_{m-1}$, with
$Q(x, i(\xv,x))=q(m)$ for $x\geq x_{m-1}$.
On the other hand, $(\xv\cup x)\prec (\xv\cup y)$ for $x<y$, hence $v(\xv\cup x)$ is nonincreasing in $x$. Thus introducing
$$h(\xv):=\sup\{x: Q(x, i(\xv,x))<v(\xv\cup x)\}$$
 we have $Q(x, i(\xv,x))<v(\xv\cup x)$ for $x<h(\xv)$, and 
 $Q(x, i(\xv,x))\geq v(\xv\cup x)$ for $x\geq h(\xv)$. 
Subject to obvious adjustments,
the definition of $h(\xv)$  makes sense for every $\xv\neq {\bf 0}$ in the untruncated 
case, and for $x_{m-1}>0$ in the truncated.

\par We are ready to show that memoryless rules are not optimal.
\begin{proposition} If $V({\cal F})<\infty$ then $V({\cal F})<V({\cal M})$.  
\end{proposition}
\proof For a memoryless rule with threshold function $f$  to be optimal,
we must have 
 $v(\bar{t}\Xv_{t})< Q(\bar{t} X_t,i(\Xv_{t-},X_t))$
for $X_t>f(t)$, and 
 $v(\bar{t}\Xv_{t})> Q(\bar{t} X_t,i(\Xv_{t-},X_t))$
for $X_t<f(t)$, because otherwise the rule can be improved.
This forces $f(t)=h(\bar{t}\xv)$, which does not hold since $h$ is not constant.

\par
To  demonstrate concretely how a memoryless rule with threshold $f$ can be improved
 let us apply the same 
idea as in \cite[Section 5]{Bruss2}. Assume $q(\infty)=\infty$.
Suppose $(t,x)$ is  above the graph of $f$, hence should be skipped by the memoryless rule.
Let $i=i(\xv,x)$ be the initial rank under history $\xv$. Varying finitely many of the components $x_r$ ($r>i$) 
 we can achieve
that the bound (\ref{hp}) be arbitrarily large while the expected loss of stopping   remains unaltered $Q(\bar{t} x,i)$.
For such $\xv$ we have $v(\bar{t}(\xv\cup x))>Q(\bar x,i(\xv,x))$ hence stopping strictly reduces the risk
on some event of positive probability.
\endpf

\par Based on the function $h:{\cal X}\to \overline{\mathbb R}_+$, we construct a predictable process
$$H_t:=h(\Xv_{t-}\setminus \{X_{1,r}: T_r< t,\,X_{1,r}<h(\Xv_{T_r-})\})~~~~~~(t\in[0,1]).$$
Let $\Yv_t$ be a thinned sequence obtained by removing
the terms in $\{\cdots\}$   from $\Xv_{t-}$,
so $H_t=h(\Yv_t)$.
Intuitively, $H_t$ is a history-dependent threshold which
depends on the configuration of atoms $\Xv_{t-}$ that arrived on $[0,t[$ and are above the curve $(H_s,\,s\in[0,t[)$.
As $t$ starts increasing from $0$, the process $H_t$ coincides  with $h(\Xv_{t-})$ as long as there are no atoms below the threshold,
while at the first moment this occurs the atom is discarded, and does not affect the future path of the process.

\vskip0.5cm\noindent
{\bf Remark}
The reason for thinning  $\cal P$ is that we wish to see $(H_t)$ as an increasing process defined for all $t$, as opposed to considering
$h(\Xv_{t-})$ killed as soon as the threshold is undershoot.
\vskip0.5cm

\par 
We list some properties of $(H_t)$ which follow directly from the definition and Lemmas \ref{cont2} and \ref{L2}
(under $\Xv_0=\varnothing$). 
\begin{lemma}
\begin{itemize}
\item[\rm (i)] $(H_t)$ is nondecreasing  on $[0,1[$\,.
\item[\rm (ii)] If $V({\cal F})<\infty$ then $H_0$ is the unique root of $Q(x,1)=v(x\cup\inftyv)$.
\item[\rm (iii)] $H_{1-}=Y_{1,m-1}$ if $q$ is truncated at $m$ and $q(m-1)<q(m)$.
\item[\rm (iv)] $H_{1-}=\infty$ if $q$ is not truncated.
\end{itemize}
\end{lemma}

\par To gain some intuition about the behaviour  of $(H_t)$  we shall gradually increase the  complexity of  loss function.
In the simplest instance of the best-choice problem, $v$ depends only on $x_1$ (see \cite[Equations (8) and (13)]{BC}) and there is an
 explicit formula for threshold
$$H_t=\min(f_b(t),Y_{t,1})~~~~~(b=0.804\cdots).$$
That is to say, as $t$ starts increasing from $0$,
$H_t$ is a deterministic {\it drift} process until it hits the level of the lowest atom above the graph.
The drift is hyperbolic due to self-similarity of $\cal P$ (Lemma \ref{selfsim}). 
After this random time, $H_t$ has a {\it flat}, which appears because it is never optimal to stop at observation with initial rank $2$ or larger.
On the first part of the path $H_t$ satisfies $Q(H_t,1)=v(\bar{t}(\Yv_t\cup H_t))$, and on the second $Q(H_t,1)<v(\bar{t}(\Yv_t\cup H_t))$.
\par If $q$ is strictly truncated at $m=3$, meaning that $q(2)<q(3)=q(\infty)$, a new effect appears.
For $t$ sufficiently small, as long as $H_t<Y_{t,1}$ each $1$-record above the threshold causes a {\it jump}, because 
$v(\bar{t}\Yv_t)$ jumps and the threshold must go up to compensate. 
Thus $(H_t)$ has both drift and jump components.
The jump locations are the $1$-record times accumulating near $0$ at rate $t^{-1}\dd t$.
As $H_t$ hits $Y_{t,1}$, there is a possible flat, then a period of deterministic drift where $Q(H_t,2)=v(\bar{t}(\Yv_t\cup H_t))$,
and finally there is a flat at some level $Y_{t,2}$ (then $Y_{t,2}=Y_{1,2}$).

\par For $q$  strictly truncated at  $m>3$, the jump locations are included  in $m-2$ record-time  processes
of atoms with initial rank at most $m-2$, there are
$m-1$ potential flats and a drift component between the flats. We do not assert that the number of flats is always exactly $m-1$, because it is not at all clear if $(H_t)$ can break a level
$Y_{t,r}$ for $r<m-1$ by jumping through it, hence sparing a flat.
\par Now suppose that $q$ is not truncated and that $H_t<\infty$ everywhere on $[0,1[$ with probability one. 
Then, outside the union of flat intervals, {\it every} arrival above $H_t$ causes a jump,
thus the set of jump locations is dense there. The number of flats may be infinite,
and outside the flats
$Q(H_t,i(\Yv_{t}, H_t))=v(\bar{t}(\Yv_t\cup H_t))$.

\par In the case of Robbins' problem, we have by linearity of the loss $Q(x,i+1)-Q(x,i)=1$  and 
$v(\xv\cup x)-v(\xv)<1$  (if $v(\xv\cup x)<\infty$).
Thus $Q(x,i(\xv,x))=v(\xv\cup x)$  implies $Q(x,i(\xv,x)+1)>v(\xv\cup x\cup x')$ for arbitrary $x'$.
But this means that  $(H_t)$ cannot cross any $Y_{t,i}$ by a jump.
It follows that $(H_t)$ has infinitely many flats at all levels $Y_{1,r}$  $(r\in\Nat)$. 
The presence of all three effects (drift, jumps and flats) and the lack of  independence of increments property
all leave a little hope for a kind of more explicit description of $(H_t)$.

\par The optimality principle requires stopping 
at atom $(t,x)$ when the history $\Xv_{t-}=\xv$ satisfies $Q(\tbar x,i(\xv,x))<v(\tbar \xv)$, whence
the following analogue of (\ref{optF}).

\begin{proposition}\label{main}
If  $V({\cal F})<\infty$ then $H_t<\infty$ a.s. for all $t<1$ and 
the stopping rule 
$$\tau^*:= \inf\{t:X_t<H_t\}~~~~ (\inf\varnothing=1)$$
is optimal in $\cal F$.
\end{proposition}
\proof
For bounded $q$ a general result \cite[Theorem 3, p. 127]{Shiryaev} is applicable since
the function $Q(x,i(\xv,x))$ is bounded and continuous on ${\cal X}\times\Nat$.
\par  Alternatively, for $q$ truncated at some $m$ one can use results of the optimal stopping theory for discrete-time processes.
To fit exactly in this framework,  focus on the sequences of $i$-records (for $i\leq m-1$) that arrive on $[\epsilon,1]$, and then let
$\epsilon\to 0$. The general bounded case follows in the limit $m\to\infty$.

\par For unbounded $q$ we use another kind of truncation (analogous to that in \cite[Section 4]{Bruss1}). For $m$ fixed, let  
$Q^{(m)}(x,i)=Q(x,\max(i,m))$ and consider the stopping problem with loss
$Q^{(m)}(\bar{t}x,i(\xv,x)$ for stopping at $(t,x)$ with history $\xv$.
This corresponds to ranking $x$ relative to at most $m$ atoms before $t$, but fully accounting all future observations
below $x$.
In this problem  it is never optimal to  stop at atom with relative rank $m$ or higher.
Indeed,  stopping at $(t,x)$ with such rank can be improved by continuing and then exploiting  any hyperbolic memoryless 
 rule
with $b<\bar{t}x$ (stopping is guaranteed before $1$ since the subgraph of $f_b$ has infinite area).
By discrete-time methods, 
 optimality of the rule $\tau^{(m)}=\inf\{t: X_t< H_t^{(m)}\}$ in the truncated problem is 
readily acquired, with 
a nondecreasing predictable process $(H_t^{(m)})$
defined through $h^{(m)}(\xv):=\sup\{x: Q^{(m)}(x, i(\xv,x))<v^{(m)}(\xv\cup x)\}$, where  $v^{(m)}$ is 
the minimum loss analogous to $v$. 
Obviously,  $Q^{(m)}(x, i(\xv,x)), v^{(m)}(\xv)$ is nondecreasing in $m$. 
\par A decisive property of this kind of truncation is that 
$Q^{(m)}(x, i)=Q(x, i)$ for $m\geq i$. This  implies that $H_t^{(m)}$ is eventually nondecreasing in $m$ and 
there exists a pointwise limit $H'_t=\lim_{m\to\infty}H^{(m)}_t$, which defines  
a legitimate stopping rule $\tau'$ as the time of the first arrival under $H'$.
Denote for shorthand $L(\tau)=\ex[Q(X_\tau,I_\tau)], L^{(m)}(\tau)=\ex[Q^{(m)}(X_\tau,I_\tau)]$ and denote 
$u, u^{(m)}$ the minimum risks (so $u=V({\cal F})$).
Trivially, $\lim_{m\to\infty} u^{(m)}\leq u$. On the other hand, by monotone convergence $L^{(m)}(\tau')\uparrow L(\tau)\geq u$.
If follows that $u^{(m)}\leq u$ and $\tau'$ is optimal.
The convergence $v^{(m)}(\xv)\uparrow v(\xv)$ is shown in the same way, from which $H'_t=H_t$ and $\tau'=\tau^*$ is optimal.
\endpf
\vskip0.5cm
\noindent
{\bf Remark.} Assumption (\ref{Ass}) limits, by the virtue of Lemma \ref{cont1}, the risks of {\it all\,} stopping rules
under various initial data, 
while we are really interested only in the properties of optimal or $\epsilon$-optimal rules. 
We feel that Proposition \ref{main} is still valid under the sole condition $V({\cal F})<\infty$, 
but history dependence makes proving this more difficult than in the analogous situation with rank rules \cite{GianSam}.

\vskip0.5cm

\par As a by-product, we have shown that the risk  in the  truncated problem with loss function $q(\min(r,m))$
converges to $V({\cal F})$. Indeed, the loss is squeezed between the loss in the modified truncated problem and the original untruncated 
loss. 
\par From the formula for the distribution of the optimal rule,  
$$\prob(\tau^*>t)=\ex\left[\exp\left(-\int_0^s H_s \dd s\right)\right]\,,$$
and arguing as in Lemma \ref{Borel} we see that
  $H_t$ cannot explode at some $t<1$ if $V({\cal F})<\infty$.

The risk can be bounded from below in the spirit of (\ref{m-rule-risk}) as
$$\ex[q(R_{\tau^*})]\geq \ex\left[\int_0^1 \exp\left(-\int_0^s H_s \dd s\right) \int_0^{H_t}Q(\bar{t}x,\phi_H(x))\dd x \right],$$
where $\phi_H(x)$ is the number of flats of $(H_t)$ below $x$.
If the loss function $q$ has the property that the flats of $(H_t)$ occur at all levels $X_{1,r},~r\in\Nat$ (like in 
Robbins' problem) the equality holds. 
The same kind of estimate is valid for every stopping rule
$\tau$ defined by means of an arbitrary nondecreasing predictable process like $(H_t)$.


\vskip0.5cm
\noindent
{\bf 4. The infinite Poisson model as a limit of finite-$n$ problems} 
To connect the finite-$n$ problem with its Poisson counterpart
it is convenient to realise iid sequence in the following way \cite{Gianini, GnFI, Denis}.
Divide the strip $[0,1]\times{\mathbb R}_+$ in $n$  vertical strips of the same width $1/n$. Let $X_j$ be the atom of $\cal P$ with
the lowest $x$-value. By properties of the Poisson process, 
$X_1,\ldots,X_n$ are iid with exponential distribution of rate $1/n$.
Note that optimal stopping of $X_1,\ldots,X_n$ is equivalent to optimal stopping of $\cal P$ with the lookback option allowing
the observer to return to any atom within a given $1/n$-strip (equivalently, at time $(j-1)/n$ to foresee the configuration of atoms up to 
time $j/n$). This embedding in $\cal P$ immediately implies $V_n({\cal F}_n)<V({\cal F})$. Moreover, as $n\to\infty$,
each  $i$-record process derived from $X_1,\ldots,X_n$ converges almost surely to the $i$-record process derived from $\cal P$.
From this one easily concludes, first for truncated then for any bounded $q$, 
 that $V_\infty({\cal F})=V({\cal F})$, where $V_\infty({\cal F})=\lim_{n\to\infty} V_n({\cal F}_n)$ as defined
in Introduction.

\par For the general $q$, the relations 
$$V_\infty({\cal F})=V({\cal F}), ~V_\infty({\cal R})=V({\cal R}), ~V_\infty({\cal M})=V({\cal M})$$
follow  (as in \cite{Assaf, BG, Bruss2, CMRS, Gianini, Mucci1}) from that in the truncated case,
by combining monotonicity of risks in the truncation parameter $m$ with the monotonicity in $n$ stated in the next lemma.
\begin{lemma} $V_n({\cal F}_n), V_n({\cal R}_n), V_n({\cal M}_n)$ are increasing with $n$.
\end{lemma}
\proof 
This all is  standard, see the references above. We only add small details to \cite[Theorem 2.4]{Assaf} for the $\cal M$-case.
Let $\tau$ be an optimal memoryless rule in the problem of size $n+1$, and let $\tau'$ be a modified memoryless strategy which
always skips the worst value $X_{n+1,n+1}$ but otherwise has the same thresholds as $\tau$.
(To apply $\tau'$ the observer must be able to recognise $X_{n+1,n+1}$ as it arrives.) Then $\tau'$ strictly improves
$\tau$ in the event that $\tau$ stops at $X_{n+1,n+1}$. On the other hand, strategy $\tau'$ performs as a mixture of memoryless rules
 in the problem of size $n$,
 because given $X_{n+1,n+1}=x$
the other $X_j$'s are iid uniform on $[0,x]$.
Therefore $V_n({\cal M}_n)<V_{n+1}({\cal M}_{n+1})$.
\endpf


\begin{thebibliography}{99}


\bibitem{Assaf} Assaf, D. and Samuel-Cahn, E. (1996) The secretary problem: minimizing the expected rank with i.i.d.
random variables, {\it Adv. Appl. Prob.} {\bf 28} 828-852.



\bibitem{BG} Berezovsky, B.A. and Gnedin, A.V. (1984)  {\it The best choice problem},
Nauka, Moscow.      


\bibitem{Bruss1} Bruss, F.T. (2005) What is known about Robbins' problem?
{\it J. Appl. Prob.} {\bf 42} 108-120.


\bibitem{Bruss2}  Bruss, F.T. and Ferguson, T.S. (1993) Minimizing the expected rank with full information,
{\it J. Appl. Prob.}  {\bf 30} 616-626.



\bibitem{Bruss3}  Bruss, F.T. and Ferguson, T.S. (1996) Half-profets and Robbins' problem of minimizing the expected rank,
{\it Springer L. Notes Stat.} {\bf 114} 1-17.


\bibitem{BrussRo} Bruss, F.T. and Rogers, L.C.G. (1991) Embedding optimal selection problems in a Poisson process,
{\it Stoch. Proc. Appl.} {\bf 38} 267-278.


\bibitem{CMRS} Chow, Y.S., Moriguti, S., Robbins, H. and Samuels, S.M. (1964) 
Optimum selection based on relative rank. (The "secretary problem"), {\it Israel J. Math.} {\bf 2} 81-90.


\bibitem{FrankSam} Frank, A. and Samuels, S.M. (1980) On an optimal stopping problem of Gusein-Zade,
{\it Stoch. proc. Appl.} {\bf 10} 299-311.


\bibitem{Gianini} Gianini, J. (1977) The infinite secretary problem as the limit of the finite problem,
{\it Ann. Prob.} {\bf 5} 636-644.


\bibitem{GianSam} Gianini, J. and Samuels, S.M. (1976) The infinite secretary problem,
{\it Ann. Prob.} {\bf 4} 418-432.


\bibitem{GnFI} Gnedin, A.V. (1996) On the full-information best-choice
problem, {\it J. Appl. Prob.} {\bf 33} 678-687.

\bibitem{BC} Gnedin, A.V. (2004) Best choice from the planar Poisson process, {\it Stoch. Proc. Appl.}
{\bf 111}, 317-354.


\bibitem{LastR} Gnedin, A.V. (2007) Recognising the last record of a sequence,
{\it Stochastics} {\bf 79} 199-210.


\bibitem{Denis} Gnedin, A.V. and Miretskiy, D.I. (2007) 
Winning rate in the full information best-choice problem,
{\it J. Appl. Prob.} (to appear).



\bibitem{KR} K{\"u}hne, R. and R{\"u}schendorf, L. (2000) Approximation of optimal stopping problems,
{\it Stoch. Proc. Appl.} {\bf 90} 301-325.



\bibitem{Mucci1} Mucci, A. (1973) Differential equations and optimal choice problems,
{\it Ann. Stat.} {\bf 1} 104-113.


\bibitem{Mucci2} Mucci, A. (1973) On a class of best-choice problems, {\it Ann. Prob.} {\bf 1} 417-427.



\bibitem{HillK} Hill, T. and Kennedy, D. (1992) Sharp inequalities for optimal stopping
with rewards based on ranks, {\it Ann. Appl. Prob.} {\bf 2} 503-517.


\bibitem{RubinS} Rubin, H. and Samuels, S.M. (1977) The finite-memory secretary problem,
{\it Ann. Prob.} {\bf 5} 627-635.


\bibitem{SamSurv} Samuels, S.M. (1991) Secretary problems. 
 In: Ghosh, B.K. and Sen, P.K. (Eds), 
{\it Handbook of sequential analysis},
Marcel Dekker, New York,   Chapter 16.


\bibitem{Shiryaev} Shiryaev, A.N. {\it Optimal stopping rules}, Springer, 1978.

\end{thebibliography}
\end{document}